\documentclass[12pt]{amsart}

\usepackage{amssymb}
\usepackage{fullpage}
\usepackage{ifthen}

\usepackage{color,soul}
\usepackage[dvipsnames]{xcolor}

\usepackage{tikz,tikz-3dplot}
\usepackage{graphicx}
\usepackage{float}

\usepackage{bm}
\usepackage{amsmath,mathtools}
 
\usepackage[normalem]{ulem}

\usepackage{mathdots}

\usepackage{url}

\newboolean{workmode}
\setboolean{workmode}{false}

\newcommand{\ifwork}[1]{\ifthenelse{\boolean{workmode}}{#1}{}}
\newcommand{\comment}[1]{}
\newcommand{\ynote}[1]{}
\newcommand{\mute}[2]{}
\newcommand{\printname}[1]{}

\ifwork{
  \renewcommand{\comment}[1]{{ \color{magenta}  \ \scriptsize{#1}\ }}
  \renewcommand{\ynote}[1]{{ \color{magenta}  \ \scriptsize{#1}\ }}
\renewcommand{\mute}[2]{{\scriptsize \ #1\ }\marginpar{\scriptsize muted}}
\renewcommand{\printname}[1]
    {{\makebox[0pt]{\hspace{-1.0in}\raisebox{8pt}{\tiny #1}}}}
}

\newcommand{\labell}[1] { \printname{#1} \label{#1} }

\numberwithin{equation}{section}

\theoremstyle{plain}

\newtheorem{Theorem}[equation]{Theorem}
\newtheorem{Proposition}[equation]{Proposition}

\newtheorem{Corollary}[equation]{Corollary}
\newtheorem{Lemma}[equation]{Lemma}

\theoremstyle{definition}
\newtheorem{Definition}[equation]{Definition}
\newtheorem{Notation}[equation]{Notation}

\theoremstyle{remark}
\newtheorem{Remark}[equation]{Remark}

\def \R {{\mathbb R}}
\def \Z {{\mathbb Z}}
\def \N {{\mathbb N}}
\def \del {\partial}
\def \ds {\displaystyle}
\DeclareMathOperator \GL {GL}

\DeclareMathOperator \oC {\mathring{C}}

\begin{document}

\title{A Pick-type theorem for halfway-lattice polygons}
\date{\today}

\author{Charlene Chu}
\email{char.chu@alumni.utoronto.ca}

\author{Yael Karshon}
\email{yaelkarshon@tauex.tau.ac.il}
\address{Tel-Aviv University, and the University of Toronto}

\begin{abstract}
Given a polygon in the plane with vertices on lattice points,
Pick's theorem expresses its area $A$ in terms of the number $I$ 
of lattice points in its interior and the number $B$ of lattice points 
in its boundary: $\ds A = I + \frac{B}{2} - 1$.  
We consider polygons in the plane whose edges are halfway between 
lattice points.   For such a polygon that is also unimodular,
we express its area $A$ in terms of the number $I$ of lattice points
in its interior and the number $V$ of its vertices: 
$\ds A = I - \frac{V}{8} + \frac{1}{2}$.
\end{abstract}


\maketitle

\setcounter{tocdepth}{1}
\tableofcontents

\section{Introduction}
\labell{sec:intro}

Given a polygon in the plane with vertices on lattice points,
Pick's theorem expresses its area $A$ in terms of 
the number $I$ of lattice points in its interior
and the number $B$ of lattice points in its boundary:
$\ds A = I + \frac{B}{2} - 1$.  
For example, the pentagon in Figure~\ref{fig:pentagon}
has an area of $\ds {\color{red} 7} + \frac{\color{blue}8}{2} - 1 = 10$.

\begin{figure}[H]
\centering
\begin{tikzpicture}
    \coordinate (A) at (1,1);
    \coordinate (B) at (3,1);
    \coordinate (C) at (4,3);
    \coordinate (D) at (2,5);
    \coordinate (E) at (0,3);
    \foreach \x in {-1,0,...,5}{
      \foreach \y in {0,1,...,6}{
        \node[draw,circle,inner sep=1pt,fill] at (1*\x,1*\y) {};
      }
    }
    \draw [,-,fill=gray!50, opacity=0.25] (A) -- (B) -- (C) -- (D) -- (E) -- (A) node [] {};
    \draw [thick,-,black] (A) -- (B) -- (C) -- (D) -- (E) -- (A) node [] {};
        \node[draw=red,circle,inner sep=1pt,fill=red] at (1,2) {};
        \node[draw=red,circle,inner sep=1pt,fill=red] at (2,2) {};
        \node[draw=red,circle,inner sep=1pt,fill=red] at (3,2) {};
        \node[draw=red,circle,inner sep=1pt,fill=red] at (1,3) {};
        \node[draw=red,circle,inner sep=1pt,fill=red] at (2,3) {};
        \node[draw=red,circle,inner sep=1pt,fill=red] at (3,3) {};
        \node[draw=red,circle,inner sep=1pt,fill=red] at (2,4) {};
        \node[draw=blue,circle,inner sep=1pt,fill=blue] at (1,1) {};
        \node[draw=blue,circle,inner sep=1pt,fill=blue] at (2,1) {};
        \node[draw=blue,circle,inner sep=1pt,fill=blue] at (3,1) {};
        \node[draw=blue,circle,inner sep=1pt,fill=blue] at (4,3) {};
        \node[draw=blue,circle,inner sep=1pt,fill=blue] at (3,4) {};
        \node[draw=blue,circle,inner sep=1pt,fill=blue] at (2,5) {};
        \node[draw=blue,circle,inner sep=1pt,fill=blue] at (1,4) {};
        \node[draw=blue,circle,inner sep=1pt,fill=blue] at (0,3) {};
\end{tikzpicture}
\caption{Pick's theorem \comment{Charlene's Figure 1}} \label{fig:pentagon}
\end{figure}

A proof of Pick's Theorem can be found in \cite[Section~2.6]{beck-robins}. 
For Pick's original paper, see \cite{pick}.

We consider polygons whose edges are halfway between lattice points.
For such a polygon that is also \emph{unimodular},
we express its area $A$ in terms of the number $I$
of lattice points in its interior and the number $V$ of its vertices:
$A = I - \frac{V}{8} + \frac{1}{2}$.
See Theorem~\ref{thm:mainTheorem}.
For example, in Figure~\ref{fig:rectangle-and-octagon},
the rectangle has an area of 
$A = {\color{red}10} 
 - \dfrac{\color{ForestGreen}4}{8} + \dfrac{1}{2} = 10$,
and the octagon has an area of 
$A = {\color{red}21} - \dfrac{\color{ForestGreen}8}{8} + \dfrac{1}{2} 
= 20 \dfrac{1}{2}$.  

{
\begin{figure}[H]
  \begin{tikzpicture}
  
    \coordinate (A) at (0.5,0.5);
    \coordinate (B) at (2.5,0.5);
    \coordinate (C) at (2.5,5.5);
    \coordinate (D) at (0.5,5.5);

    \coordinate (A1) at (0.6,0.5);
    \coordinate (A2) at (0.5,0.6);
    
    \coordinate (B1) at (2.5,0.6);
    \coordinate (B2) at (2.4,0.5);
    
    \coordinate (C1) at (2.4,5.5);
    \coordinate (C2) at (2.5,5.4);
    
    \coordinate (D1) at (0.6,5.5);
    \coordinate (D2) at (0.5,5.4);

    \foreach \x in {0,1,...,3}{
      \foreach \y in {0,1,...,6}{
        \node[draw,circle,inner sep=1pt,fill] at (1*\x,1*\y) {};

        \node[draw=red,circle,inner sep=1pt,fill=red] at (1,1) {};
        \node[draw=red,circle,inner sep=1pt,fill=red] at (2,1) {};
        \node[draw=red,circle,inner sep=1pt,fill=red] at (1,2) {};
        \node[draw=red,circle,inner sep=1pt,fill=red] at (2,2) {};
        \node[draw=red,circle,inner sep=1pt,fill=red] at (1,3) {};
        \node[draw=red,circle,inner sep=1pt,fill=red] at (2,3) {};
        \node[draw=red,circle,inner sep=1pt,fill=red] at (1,4) {};
        \node[draw=red,circle,inner sep=1pt,fill=red] at (2,4) {};
        \node[draw=red,circle,inner sep=1pt,fill=red] at (1,5) {};
        \node[draw=red,circle,inner sep=1pt,fill=red] at (2,5) {};

      }
    }
    \draw [,-,fill=gray!50, opacity=0.25,gray!50, opacity=0.25] (A) -- (B) -- (C) -- (D) -- (A)--cycle;
    \draw [thick,-,black] (A)
        -- (B) -- (C) -- (D) -- (A) node [above left] {};
        
    \draw [ultra thick,-,ForestGreen] (A1) -- (A) -- (A2) node [above left] {};
    \draw [ultra thick,-,ForestGreen] (B1) -- (B) -- (B2) node [above left] {};
    \draw [ultra thick,-,ForestGreen] (C1) -- (C) -- (C2) node [above left] {};
    \draw [ultra thick,-,ForestGreen] (D1) -- (D) -- (D2) node [above left] {};
  \end{tikzpicture}
\qquad \qquad
%
%
    \begin{tikzpicture}
    
    \coordinate (A) at (2,0.5);
    \coordinate (B) at (4,0.5);
    \coordinate (C) at (5.5,2);
    \coordinate (D) at (5.5,4);
    \coordinate (E) at (4,5.5);
    \coordinate (F) at (2,5.5);
    \coordinate (G) at (0.5,4);
    \coordinate (H) at (0.5,2);
        
    \coordinate (A1) at (2.1,0.5);
    \coordinate (B1) at (4.07,0.57);
    \coordinate (C1) at (5.5,2.1);
    \coordinate (D1) at (5.43,4.07);
    \coordinate (E1) at (3.9,5.5);
    \coordinate (F1) at (1.93,5.43);
    \coordinate (G1) at (0.5,3.9);
    \coordinate (H1) at (0.57,1.93);
    
    \coordinate (A2) at (1.93,0.57);
    \coordinate (B2) at (3.9,0.5);
    \coordinate (C2) at (5.43,1.93);
    \coordinate (D2) at (5.5,3.9);
    \coordinate (E2) at (4.07,5.43);
    \coordinate (F2) at (2.1,5.5);
    \coordinate (G2) at (0.57,4.07);
    \coordinate (H2) at (0.5,2.1);
    
    \foreach \x in {0,1,...,6}{
      \foreach \y in {0,1,...,6}{
        \node[draw,circle,inner sep=1pt,fill] at (1*\x,1*\y) {};
            
        \node[draw=red,circle,inner sep=1pt,fill=red] at (2,1) {};
        \node[draw=red,circle,inner sep=1pt,fill=red] at (3,1) {};
        \node[draw=red,circle,inner sep=1pt,fill=red] at (4,1) {};
        
        \node[draw=red,circle,inner sep=1pt,fill=red] at (1,2) {};
        \node[draw=red,circle,inner sep=1pt,fill=red] at (2,2) {};
        \node[draw=red,circle,inner sep=1pt,fill=red] at (3,2) {};
        \node[draw=red,circle,inner sep=1pt,fill=red] at (4,2) {};
        \node[draw=red,circle,inner sep=1pt,fill=red] at (5,2) {};
        
        \node[draw=red,circle,inner sep=1pt,fill=red] at (1,3) {};
        \node[draw=red,circle,inner sep=1pt,fill=red] at (2,3) {};
        \node[draw=red,circle,inner sep=1pt,fill=red] at (3,3) {};
        \node[draw=red,circle,inner sep=1pt,fill=red] at (4,3) {};
        \node[draw=red,circle,inner sep=1pt,fill=red] at (5,3) {};
        
        \node[draw=red,circle,inner sep=1pt,fill=red] at (1,4) {};
        \node[draw=red,circle,inner sep=1pt,fill=red] at (2,4) {};
        \node[draw=red,circle,inner sep=1pt,fill=red] at (3,4) {};
        \node[draw=red,circle,inner sep=1pt,fill=red] at (4,4) {};
        \node[draw=red,circle,inner sep=1pt,fill=red] at (5,4) {};
        
        \node[draw=red,circle,inner sep=1pt,fill=red] at (2,5) {};
        \node[draw=red,circle,inner sep=1pt,fill=red] at (3,5) {};
        \node[draw=red,circle,inner sep=1pt,fill=red] at (4,5) {};
        
      }
    }
    \draw [,-,fill=gray!50, opacity=0.25,gray!50, opacity=0.25] (A) 
        -- (B) -- (C) -- (D) -- (E) --(F) -- (G)--(H)--(A)--cycle;
    \draw [thick,-,black] (A)
        -- (B) -- (C) -- (D) -- (E) --(F) -- (G)--(H)--(A) node [] {};
        
    \draw [ultra thick,-,ForestGreen] (A1) -- (A) -- (A2) node [] {};
    \draw [ultra thick,-,ForestGreen] (B1) -- (B) -- (B2) node [] {};
    \draw [ultra thick,-,ForestGreen] (C1) -- (C) -- (C2) node [] {};
    \draw [ultra thick,-,ForestGreen] (D1) -- (D) -- (D2) node [] {};
    \draw [ultra thick,-,ForestGreen] (E1) -- (E) -- (E2) node [] {};
    \draw [ultra thick,-,ForestGreen] (F1) -- (F) -- (F2) node [] {};
    \draw [ultra thick,-,ForestGreen] (G1) -- (G) -- (G2) node [] {};
    \draw [ultra thick,-,ForestGreen] (H1) -- (H) -- (H2) node [] {};
        
    \end{tikzpicture}
\caption{Our Pick-type theorem \comment{Charlene's Figure 16+17}}
\label{fig:rectangle-and-octagon} 
\end{figure}
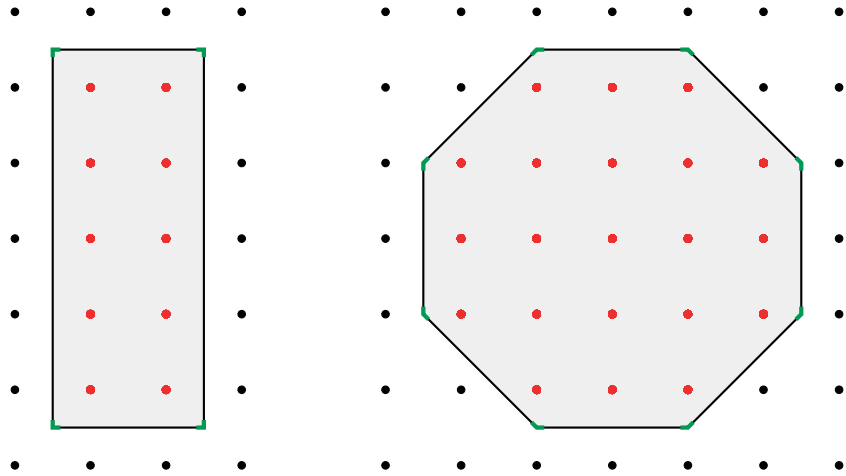
}

Being unimodular means that adjacent edges are generated
by vectors that form a basis to the lattice;
equivalently, their inner normals can be chosen to be a basis
to the (dual) lattice.
For details, see Section~\ref{sec:polygons}.

For polygons that are not unimodular, the situation is more complicated.
For example, the two kites in Figure~\ref{fig:kites-and-diamonds} on the left 
are non-unimodular.
They have the same area and the same primitive inner normal vectors,
but the numbers of lattice points in the interior are different:
Kite $A$ has two lattice points in its interior
while Kite $B$ has no lattice points in its interior.
A similar example can be seen with the two diamonds 
in Figure~\ref{fig:kites-and-diamonds} on the right.
From these examples we conclude that an analogous formula
in the non-unimodular case cannot only involve the primitive inner
normal vectors and the number of lattice points in the interior.

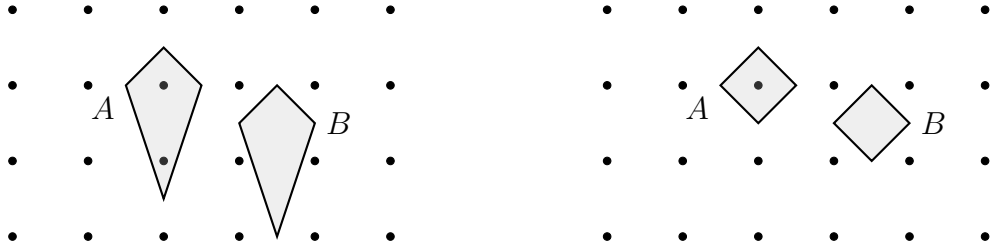
\begin{figure}[H]
\centering
  \begin{tikzpicture}
  
    \coordinate (A) at (1,1.5);
    \coordinate (B) at (1.5,0);
    \coordinate (C) at (2,1.5);
    \coordinate (D) at (1.5,2);
    
    \coordinate (E) at (-0.5,2);
    \coordinate (F) at (0,0.5);
    \coordinate (G) at (0.5,2);
    \coordinate (H) at (0,2.5);

    \foreach \x in {-2,-1,...,3}{
      \foreach \y in {0,1,...,3}{
        \node[draw,circle,inner sep=1pt,fill] at (1*\x,1*\y) {};
        
      }
    }
    
    \draw [,-,fill=gray!50, opacity=0.25,gray!50, opacity=0.25] (A) 
        -- (B) -- (C) -- (D) -- (A) -- cycle;
    \draw [thick,-,black] (C) -- (D) -- (A) -- (B) -- (C) -- cycle node [right] {$B$};
    \draw [,-,fill=gray!50, opacity=0.25,gray!50, opacity=0.25] (E) 
        -- (F) -- (G) -- (H) -- (E) -- cycle;
    \draw [thick,-,black] (E) -- (F) -- (G) -- (H) -- (E) --cycle node [below left] {$A$};

  \end{tikzpicture}
\qquad \qquad \qquad 
{
  \begin{tikzpicture}
  
    \coordinate (A) at (1,1.5);
    \coordinate (B) at (1.5,1);
    \coordinate (C) at (2,1.5);
    \coordinate (D) at (1.5,2);

    \coordinate (E) at (-0.5,2);
    \coordinate (F) at (0,1.5);
    \coordinate (G) at (0.5,2);
    \coordinate (H) at (0,2.5);

    \foreach \x in {-2,-1,...,3}{
      \foreach \y in {0,1,...,3}{
        \node[draw,circle,inner sep=1pt,fill] at (1*\x,1*\y) {};
      }
    }
    
    \draw [,-,fill=gray!50, opacity=0.25,gray!50, opacity=0.25] (A) 
        -- (B) -- (C) -- (D) -- (A) -- cycle;
    \draw [thick,-,black] (C) -- (D) -- (A) -- (B) -- (C) -- cycle node [right] {$B$};
    \draw [,-,fill=gray!50, opacity=0.25,gray!50, opacity=0.25] (E) 
        -- (F) -- (G) -- (H) -- (E) -- cycle;
    \draw [thick,-,black] (E) -- (F) -- (G) -- (H) -- (E) --cycle node [below left] {$A$};

\end{tikzpicture}
}
\caption{Non-unimodular quadrilaterals \comment{Charlene's Figures 18 and 19}} 
\label{fig:kites-and-diamonds}
\end{figure}

\bigskip

\begin{Remark}
The motivation for considering halfway-lattice polygons
(or higher dimensional halfway-lattice polytopes)
comes from geometric quantization.
Unimodular lattice polytopes occur as momentum map images
of symplectic toric manifolds that are pre-quantizable
in the classical sense of Kostant and Souriau.
%
%
Unimodular halfway-lattice polytopes occur as momentum map images 
of symplectic toric manifold that are pre-quantizable
in the \mbox{spin-c} or metaplectic-c sense \cite{grossberg-karshon,vaughan}.
%

For nice accounts of symplectic toric manifolds,
we suggest Audin's book \cite{audin}, and, in real dimension four, 
Camarneiro's MSc thesis \cite{camarneiro-CdS-thesis}.
 
We expect Euler-Maclaurin formulas for polytopes
\cite{khovanskii-pukhlikov,guillemin,brion-vergne}
(also see \cite{karshon-sternberg-weitsman-PNAS, 
karshon-sternberg-weitsman-remainder,
karshon-sternberg-weitsman-exact})
to have analogues for halfway-lattice polytopes
in which the Todd function is replaced by the A-roof function.
In the unimodular case, this was recently confirmed 
by Duarte and Caviedes Castro \cite{duarte-ACC}.
\end{Remark}

\subsection*{Acknowledgements}
This paper evolved from an M.Sc.\ project of Charlene Chu 
under the supervision of Yael Karshon at the University of Toronto.
An earlier version appeared as a research poster
for which Chu received the Best Poster Presentation Award
at the 
University of Toronto Mississauga 2016 Graduate Research Colloquium.

We would like to thank Sinai Robins and Aleksandr Kolpakov
for helpful discussions when this work was carried out.
Karshon would like to thank Ana Cannas da Silva, Alexander Caviedes Castro, 
and Oded Elisha for more recent helpful discussions.

This research was partially funded by the Natural Science and Engineering
Research Council of Canada.

\section{Unimodular halfway-lattice polygons}
\labell{sec:polygons}

A \textbf{polygon} is a subset of $\R^2$ 
that is a convex hull of finitely many points
and whose interior is non-empty.
Equivalently, it is a bounded subset of $\R^2$ with non-empty interior
that can be written in the form
\begin{equation} \labell{intersection}
 \{ x \in \R^2 \ | \ \left< x , v_i \right> + \lambda_i \geq 0
 \text{ for $i=1,\ldots,n$ } \} 
\end{equation}
for some non-zero vectors $v_1,\ldots,v_n \in \R^2$ 
and numbers $\lambda_1,\ldots,\lambda_n \in \R$.
This equation expresses the polygon as an intersection 
of the half-planes 
$H_j = \{ x \in \R^2 \ | \ \left< x,v_i \right> + \lambda_i \geq 0 \}$.
Such an expression for which $n$ is minimal is unique up to a permutation
of the half-planes.  
The \textbf{edges} of the polygon are then its intersection with the lines
$\{ \left< x, v_i \right> + \lambda_i = 0 \}$,
and the $v_i$ are \emph{inner normal vectors} to the edges.
We can choose the indexing so that the edges are consecutive
and in counterclockwise order; the indexing is then unique
up to a cyclic permutation.  
The polygon is \textbf{rational} if the slopes of its edges are rational;
equivalently, if the inner normal vectors $v_i$ 
can be chosen to be in the lattice $\Z^2$.
In this case, we can choose them so that each $v_i$ is \textbf{primitive},
i.e.,
not a multiple of another lattice element by an integer greater than 1.
With this choice of $v_i$, the polygon is \textbf{unimodular} if 
$$ \det (v_i, v_{i+1}) = 1$$
for every $i=1,\ldots,n$, where we set $v_{n+1} = v_1$,
and it is a \textbf{halfway-lattice} polygon if
the $\lambda_i$ are half-integers (i.e., integers plus $\frac{1}{2}$).
In particular, each edge of the polygon lies on a line
of rational slope that does not meet any lattice points.

The affine transformations of $\R^2$ 
that take the lattice $\Z^2$ onto itself are those of the form
$x \mapsto Ax + b$ for $A \in \GL(2,\Z)$ and for $b \in \Z^2$.
\ ($\GL(2,\Z)$ is the group of $2 \times 2$ matrices 
with determinant $\pm 1$; equivalently, the invertible $2 \times 2$ matrices
with integer entries and whose inverse has integer entries.)
Primitive lattice vectors $u,v$ satisfy $\det(u,v) = 1$
if and only if there exists such a transformation
that takes them to the standard basis elements $\ds \binom{1}{0}$
and $\ds \binom{0}{1}$, respectively.
Two polygons are \textbf{lattice congruent} if there exists
such a transformation that carries one of them onto the other.
Polygons that are lattice congruent have the same area,
the same number of vertices, and the same number of lattice points
in their interior.
This follows from the following result.

\begin{Proposition} \labell{prop:Ax+bProperties}
Let $P$ be a polygon.   Let $f \colon \R^2 \to \R^2$
be a map of the form $x \mapsto Ax +b$ with $A \in \GL(2,\Z)$ and $b \in \Z^2$.
Let $Q = f(P)$.
\begin{itemize}
\item \labell{item-linear}
$Q$ is a polygon.
The map $f$ carries the interior of $P$ to the interior of $Q$,
the vertices of $P$ to those of $Q$,
and the edges of $P$ to those of $Q$.
The linear map $(A^{-1})^T \colon \R^2 \to \R^2$
maps consecutive inner normal vectors to edges of $P$ 
to consecutive inner normal vectors to edges of $Q$.
\item \labell{item-area}
The area of $Q$ is the same as that of $P$.
\item \labell{item-lattice}
If $P$ is rational, then so is $Q$, and the map $(A^{-1})^T$
maps the primitive inner normal vectors to edges of $P$ to those of $Q$.
If $P$ is unimodular, then so is $Q$.
If $P$ is a halfway-lattice polygon, then so is $Q$.
\end{itemize}
\end{Proposition}

\begin{proof}[Sketch of proof]
The first item is true because $f$ is a linear isomorphism
followed by a translation.
The second is true because $|\det A| = 1$.
The third is true because $A$ maps the lattice onto the lattice.
\end{proof}

\section{Unimodular halfway-lattice triangles and trapezoids}
\labell{sec:triangles}

{
\smallskip

A \textbf{triangle} is a polygon with three edges.
A \textbf{standard triangle} 
is a (unimodular) triangle whose primitive inner normal vectors are 
$$ \binom{1}{0} , \ 
   \binom{0}{1} , \ \text{ and }
   \binom{-1}{-1}. $$
For example, see Figure \ref{fig:triangle}. 

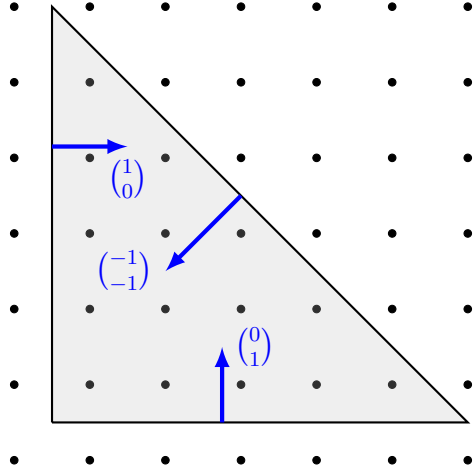
\begin{figure}[H]
\centering
  \begin{tikzpicture}
  
    \coordinate (A) at (0.5,0.5);
    \coordinate (B) at (6,0.5);
    \coordinate (C) at (0.5,6);
    
    \coordinate (v1a) at (0.5,4.15);    
    \coordinate (v1b) at (1.5,4.15);
    \coordinate (v2a) at (2.75,0.5);
    \coordinate (v2b) at (2.75,1.5);
    \coordinate (v3a) at (3,3.5);
    \coordinate (v3b) at (2,2.5);
 
    \foreach \x in {0,1,...,6}{
      \foreach \y in {0,1,...,6}{
        \node[draw,circle,inner sep=1pt,fill] at (1*\x,1*\y) {};
      }
    }
    
    \draw [,-,fill=gray!50, opacity=0.25,gray!50, opacity=0.25] (A) -- (B) -- (C) -- (A)--cycle;
    \draw [thick,-,black] (A) -- (B) -- (C) -- (A) node [above left] {};
 
\draw [ultra thick,-latex,blue] (v1a) -- (v1b) node [below] 
      {$\binom{1}{0}$};
\draw [ultra thick,-latex,blue] (v2a) -- (v2b) node [right] 
      {$\binom{0}{1}$};
\draw [ultra thick,-latex,blue] (v3a) -- (v3b) node [left] 
      {$\binom{-1}{-1}$};
 
  \end{tikzpicture}
\caption{A standard triangle \comment{Charlene's Figure 3: a standard triangle}} 
\label{fig:triangle}
\end{figure}

\begin{Proposition} \labell{prop:standardTriangleCongruence}
Every unimodular triangle is lattice congruent to a standard triangle.
\end{Proposition}

\begin{proof}
Let $P$ be a unimodular halfway-lattice triangle with primitive 
inner normal vectors $v_{1}, v_{2}, v_{3}$, in counterclockwise 
consecutive order.
Because $\det (v_1,v_2) = 1$,  up to lattice congruence, we may assume
that $\ds v_1=\binom{1}{0}$ and $\ds v_2=\binom{0}{1}$. 
Write $\ds v_3=\binom{x_3}{y_3}$.
Then $x_3 = -\det (v_2,v_3) = -1$, and $y_3 = -\det (v_3,v_1) = -1$.
\end{proof}
}

\smallskip
A \textbf{quadrilateral} is a polygon with four edges.
A \textbf{trapezoid} is a quadrilateral with two parallel edges;
equivalently, it is a quadrilateral that has two opposite inner normal vectors.
A \textbf{standard trapezoid} is a (unimodular) quadrilateral
whose primitive inner normal vectors are 
$$ \binom{1}{0} , \ 
   \binom{0}{1} , \ 
   \binom{-1}{k} , \ \text{ and }
   \binom{0}{-1} $$
for some non-negative integer $k$.
For example, see Figure \ref{fig:hirzebruch}, where $k=2$.

{

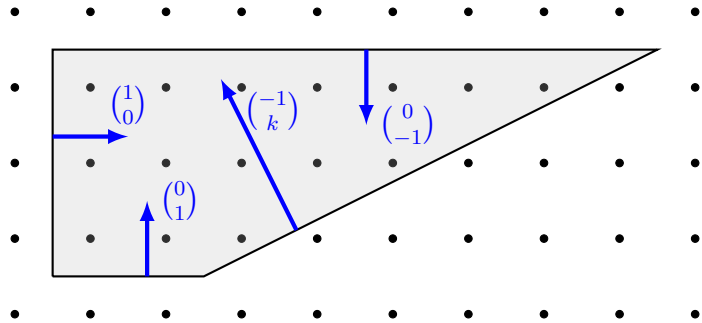
\begin{figure}[H]
\centering
 
  \begin{tikzpicture}
  
    \coordinate (A) at (0.5,0.5);
    \coordinate (B) at (2.5,0.5);
    \coordinate (C) at (8.5,3.5);
    \coordinate (D) at (0.5,3.5);
    
    \coordinate (v1a) at (0.5,2.35);    
    \coordinate (v1b) at (1.5,2.35);
    \coordinate (v2a) at (1.75,0.5);
    \coordinate (v2b) at (1.75,1.5);
    \coordinate (v3a) at (3.725,1.1125);
    \coordinate (v3b) at (2.725,3.1125);
    \coordinate (v4a) at (4.65,3.5);
    \coordinate (v4b) at (4.65,2.5);
 
    \foreach \x in {0,1,...,9}{
      \foreach \y in {0,1,...,4}{
        \node[draw,circle,inner sep=1pt,fill] at (1*\x,1*\y) {};
      }
    }
    
    \draw [,-,fill=gray!50, opacity=0.25,gray!50, opacity=0.25] (A) -- (B) -- (C) -- (D) -- (A)--cycle;
    \draw [thick,-,black] (A) -- (B) -- (C) -- (D) -- (A) node [above left] {};
\draw [ultra thick,-latex,blue] (v1a) -- (v1b) node [above] 
        {$\binom{1}{0}$};
\draw [ultra thick,-latex,blue] (v2a) -- (v2b) node [right] 
        {$\binom{0}{1}$};
\draw [ultra thick,-latex,blue] (v3a) -- (v3b) node [below right] 
        {$\ \binom{-1}{k}$};
\draw [ultra thick,-latex,blue] (v4a) -- (v4b) node [right] 
        {$\binom{0}{-1}$};
    
  \end{tikzpicture}
\caption{A standard trapezoid \comment{Charlene's Figure 5: a standard trapezoid}}
\label{fig:hirzebruch}
\end{figure}

In Corollary~\ref{prop:hirzebruchTrapezoidCongruence} we show that 
every unimodular quadrilateral is lattice congruent to a standard trapezoid.
Its proof relies on the following special case:

\begin{Proposition} \labell{prop:parallel-then-standard}
Every unimodular trapezoid is lattice congruent to a standard trapezoid.
\end{Proposition}

\begin{proof}
Let $P$ be a unimodular trapezoid, and let $\Sigma^1$ be the set
of primitive inner normal vectors to its edges.
By assumption, $\Sigma^1$ contains two opposite vectors.
Enumerate the elements of $\Sigma^1$ 
in counterclockwise consecutive order as $v_1,v_2,v_3,v_4$ 
such that $v_2$ and $v_4$ are opposite;
we can do this 
because opposite vectors cannot be normals to adjacent edges.
Up to lattice congruence, we may assume that 
$\displaystyle v_1 = \binom{1}{0}$ and $\displaystyle v_2 = \binom{0}{1}$.
So 
$\displaystyle v_4= \binom{0}{-1}$,
and 
the condition $\det(v_2,v_3)=1$ implies that 
$\displaystyle v_3=\binom{-1}{k}$ for some integer $k$.
If $k$ is negative, then applying the lattice congruence
$\displaystyle \binom{x}{y} \mapsto \binom{x}{-y}$
yields a standard trapezoid.
\end{proof}

\begin{Proposition} \labell{prop:area4-parallel}
Let $P$ be a unimodular halfway-lattice trapezoid.
Let $A$ be the area of $P$, and let $I$ be the number of lattice points
in its interior.  Then 
$$A=I.$$
\end{Proposition}

\begin{proof}
By Propositions~\ref{prop:parallel-then-standard}
and~\ref{prop:Ax+bProperties}, we may assume without loss of generality 
that $P$ is a standard trapezoid,
as in Figure~\ref{fig:hirzebruchTrapezoidProof}.

{
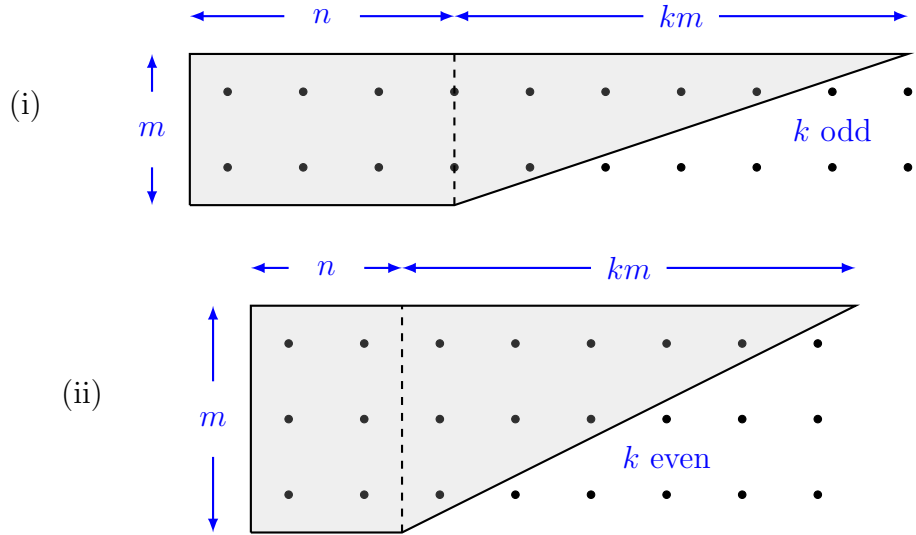
\begin{figure}[H]

\centering

  (i) \qquad \begin{tabular}{c}
  \begin{tikzpicture}

    \foreach \x in {1,2,...,10}{
      \foreach \y in {1,2,...,2}{
        \node[draw,circle,inner sep=1pt,fill] at (1*\x,1*\y) {};
      }
    }

    \coordinate (A) at (0.5,0.5);
    \coordinate (B) at (4,0.5);
    \coordinate (C) at (10,2.5);
    \coordinate (D) at (0.5,2.5);
    \draw [,-,fill=gray!50, opacity=0.25,gray!50, opacity=0.25] (A) -- (B) -- (C) -- (D) -- (A)--cycle;
    \draw [thick,-,black] (A) -- (B) -- (C) -- (D) -- (A) node [above left] {};


    \coordinate (v1a) at (1.75,3);
    \coordinate (v1b) at (0.5,3);
    \coordinate (v2a) at (2.75,3);
    \coordinate (v2b) at (4,3);
    \coordinate (n) at (2.25,3);
    \draw [thick,-latex,blue] (v1a) -- (v1b) node [above left] {};
    \draw [thick,-latex,blue] (v2a) -- (v2b) node [below left] {};
    \draw [thick,-,blue] (n) node [] {$n$};
    \coordinate (v3a) at (0,2);
    \coordinate (v3b) at (0,2.5);
    \coordinate (v4a) at (0,1);
    \coordinate (v4b) at (0,0.5);
    \coordinate (m) at (0,1.5);
    \draw [thick,-latex,blue] (v3a) -- (v3b) node [right] {};
    \draw [thick,-latex,blue] (v4a) -- (v4b) node [above right] {};
    \draw [thick,-,blue] (m) node [] {$m$};

    \coordinate (v5a) at (6.5,3);
    \coordinate (v5b) at (4,3);
    \coordinate (v6a) at (7.5,3);
    \coordinate (v6b) at (10,3);
    \coordinate (km) at (7,3);
    \draw [thick,-latex,blue] (v5a) -- (v5b) node [right] {};
    \draw [thick,-latex,blue] (v6a) -- (v6b) node [right] {};
    \draw [thick,-,blue] (km) node [] {$km$};

    \coordinate (k) at (9,1.5);
    \draw [thick,-,blue] (k) node [] {$k$ odd};

    \coordinate (vline1a) at (4,0.5);
    \coordinate (vline1b) at (4,2.5);
    \draw [thick,dashed,black] (vline1a) -- (vline1b) node [] {};



  \end{tikzpicture}
  \end{tabular}

\bigskip

  (ii) \qquad \begin{tabular}{c}
  \begin{tikzpicture}

    \foreach \x in {1,2,...,8}{
      \foreach \y in {1,2,...,3}{
        \node[draw,circle,inner sep=1pt,fill] at (1*\x,1*\y) {};
      }
    }

    \coordinate (A) at (0.5,0.5);
    \coordinate (B) at (2.5,0.5);
    \coordinate (C) at (8.5,3.5);
    \coordinate (D) at (0.5,3.5);
    \draw [,-,fill=gray!50, opacity=0.25,gray!50, opacity=0.25] (A) -- (B) -- (C) -- (D) -- (A)--cycle;
    \draw [thick,-,black] (A) -- (B) -- (C) -- (D) -- (A) node [above left] {};

    \coordinate (v1a) at (1,4);
    \coordinate (v1b) at (0.5,4);
    \coordinate (v2a) at (2,4);
    \coordinate (v2b) at (2.5,4);
    \coordinate (n) at (1.5,4);
    \draw [thick,-latex,blue] (v1a) -- (v1b) node [above left] {};
    \draw [thick,-latex,blue] (v2a) -- (v2b) node [below left] {};
    \draw [thick,-,blue] (n) node [] {$n$};
    \coordinate (v3a) at (0,2.5);
    \coordinate (v3b) at (0,3.5);
    \coordinate (v4a) at (0,1.5);
    \coordinate (v4b) at (0,0.5);
    \coordinate (m) at (0,2.0);
    \draw [thick,-latex,blue] (v3a) -- (v3b) node [right] {};
    \draw [thick,-latex,blue] (v4a) -- (v4b) node [above right] {};
    \draw [thick,-,blue] (m) node [] {$m$};

    \coordinate (v5a) at (5,4);
    \coordinate (v5b) at (2.5,4);
    \draw [thick,-latex,blue] (v5a) -- (v5b) node [right] {};
    \coordinate (v6a) at (6,4);
    \coordinate (v6b) at (8.5,4);
    \draw [thick,-latex,blue] (v6a) -- (v6b) node [right] {};
    \coordinate (km) at (5.5,4);
    \draw [thick,-,blue] (km) node [] {$km$};

    \coordinate (k) at (6,1.5);
    \draw [thick,-,blue] (k) node [] {$k$ even};

    \coordinate (vline1a) at (2.5,0.5);
    \coordinate (vline1b) at (2.5,3.5);
    \draw [thick,dashed,black] (vline1a) -- (vline1b) node [] {};

  \end{tikzpicture}\end{tabular}
\caption{Two standard trapezoids 
\comment{Charlene's Figure 6}} \label{fig:hirzebruchTrapezoidProof}
\end{figure}
}

The $x$-coordinates of the top-right and bottom-right vertices 
are either both integers or both half-integers 
(i.e., integers plus $\frac{1}{2}$),
depending on whether the non-negative integer $k$ 
in the primitive inner normal vector 
$\ds \binom{-1}{\phantom{-}k}$ of $P$ is odd or even.
Either way, two copies of the standard trapezoid fit together
 --- after applying a lattice congruence to one of them --- into a 
halfway-lattice rectangle, as in Figure~\ref{fig:fit-into-rectangle}.
The area of the rectangle is $2A$, and the number of lattice points
in its interior is $2I$. 
By subdividing the rectangle into squares
as in Figure~\ref{fig:rectangle-into-squares}, we see that 
the area of the rectangle is equal to the number of lattice points
in its interior. That is, $2A=2I$, and thus $A=I$.

{
\begin{figure}[H]

\centering

\qquad \begin{tabular}{c}
\begin{tikzpicture}

    \foreach \x in {1,2,...,10}{
      \foreach \y in {1,2,...,3}{
        \node[draw,circle,inner sep=1pt,fill] at (1*\x,1*\y) {};
      }
    }

    \coordinate (A) at (0.5,0.5);
    \coordinate (B) at (2.5,0.5);
    \coordinate (C) at (8.5,3.5);
    \coordinate (D) at (0.5,3.5);
    \draw [,-,fill=gray!50, opacity=0.25,gray!50, opacity=0.25] (A) -- (B) -- (C) -- (D) -- (A)--cycle;
    \draw [thick,-,black] (A) -- (B) -- (C) -- (D) -- (A) node [above left] {};
\coordinate (E) at (10.5,0.5);
\coordinate (F) at (10.5,3.5);
    \draw [,-,fill=gray!50, opacity=0.25,gray!50, opacity=0.25] (B) -- (E) -- (F) -- (C)--cycle;
    \draw [thick,-,black] (B) -- (E) -- (F) -- (C) node [above left] {};
%

  \end{tikzpicture}\end{tabular}
\caption{Double of a standard trapezoid} \label{fig:fit-into-rectangle}
\end{figure}
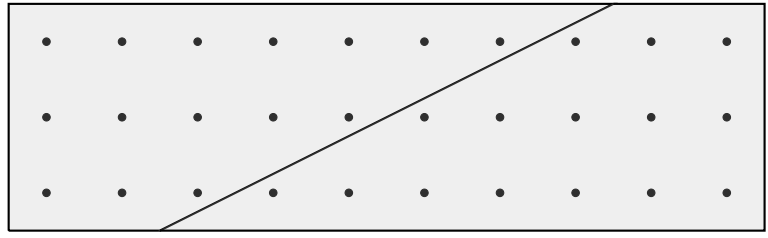
}

{
\begin{figure}[H]

\centering

\qquad \begin{tabular}{c}
\begin{tikzpicture}

    \foreach \x in {1,2,...,10}{
      \foreach \y in {1,2,...,3}{
        \node[draw,circle,inner sep=1pt,fill] at (1*\x,1*\y) {};
      }
    }

    \coordinate (A) at (0.5,0.5);
    \coordinate (B) at (10.5,0.5);
    \coordinate (C) at (10.5,3.5);
    \coordinate (D) at (0.5,3.5);
    \draw [,-,fill=gray!50, opacity=0.25,gray!50, opacity=0.25] (A) -- (B) -- (C) -- (D) -- (A)--cycle;
    \draw [thick,-,black] (A) -- (B) -- (C) -- (D) -- (A) node [above left] {};
    \foreach \x in {1,2,...,9}{
    \draw [thick,dashed,black] (0.5+\x,0.5) -- (0.5+\x,3.5) node [] {};
    }
    \foreach \y in {1,2}{
    \draw [thick,dashed,black] (0.5,0.5+\y) -- (10.5,0.5+\y) node [] {};
    }

  \end{tikzpicture}\end{tabular}
\caption{Unimodular halfway-lattice rectangle} 
\label{fig:rectangle-into-squares}
\end{figure}
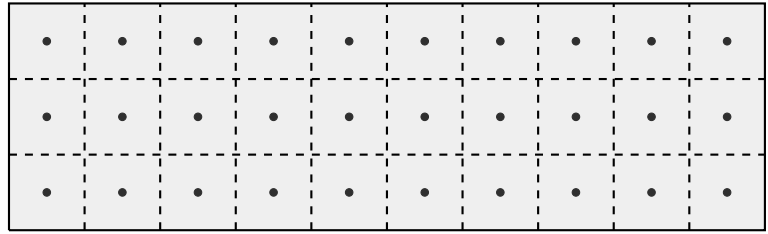
}
\end{proof}

{
\begin{Proposition}  \labell{prop:area3}
Let $P$ be a unimodular halfway-lattice triangle.
Let $A$ be its area, and let $I$ be the number of lattice points 
in its interior.  Then $$A = I + \frac{1}{8}.$$
\end{Proposition}

\begin{proof}
{
By Propositions~\ref{prop:standardTriangleCongruence}
and~\ref{prop:Ax+bProperties}, we may assume 
without loss of generality that the halfway-lattice triangle $P$ is standard,
and that it's given by the equations $x \geq \frac{1}{2}$, 
$y \geq \frac{1}{2}$, and $x+y \leq k+\frac{1}{2}$ for some $k \in \N$,
as in Figure~\ref{fig:triangle}.
Intersecting it with the half-plane cut out by the inequality
$\ds x+y \geq \frac{1}{2}$ yields a unimodular halfway-lattice trapezoid; see
Figure~\ref{fig:standardTriangleProof2}. 
The remaining triangle is a standard halfway-lattice triangle 
of area $\ds \frac{1}{8}$ and with no interior lattice points.
So the area of the trapezoid is $\ds A - \frac{1}{8}$,
and the number of lattice points in its interior is $I$.
By Proposition~\ref{prop:area4-parallel}, $\ds A - \frac{1}{8} = I$.
The result follows.

\medskip
{
\begin{figure}[H]
\centering
  \begin{tikzpicture}

    \foreach \x in {0,1,2,...,6}{
      \foreach \y in {0,1,2,...,6}{
        \node[draw,circle,inner sep=1pt,fill] at (1*\x,1*\y) {};
      }
    }

    \coordinate (A) at (0.5,0.5);
    \coordinate (B) at (6,0.5);
    \coordinate (C) at (0.5,6);
    \draw [,-,fill=gray!50, opacity=0.25,gray!50, opacity=0.25] (A) -- (B) -- (C) -- (A)--cycle;
    \draw [thick,-,black] (A) -- (B) -- (C) -- (A) node [above left] {};

    \coordinate (D) at (1,0.5);
    \coordinate (E) at (0.5,1);
    \draw [thick,black] (D) -- (E) node [] {};
    \coordinate(F) at (1.5,0);
    \coordinate(G) at (0,1.5);
    \draw [thick,dashed,black] (D) -- (F) node [] {};
    \draw [thick,dashed,black] (E) -- (G) node [] {};

  \end{tikzpicture}
\caption{Triangle to trapezoid} \label{fig:standardTriangleProof2}
\end{figure}
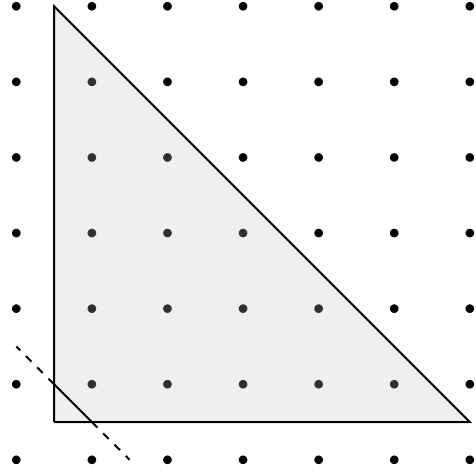
}}

\end{proof}
}

\section{Two opposite inner normal vectors}
\labell{sec:opposite}

Proposition~\ref{prop:area4-parallel} shows that our Pick-type theorem
holds for unimodular polygons with $n=4$ edges
whose set of primitive inner normals contains two opposite vectors.
In Proposition~\ref{no opposite implies triangle} below we recall 
that \emph{every} unimodular polygon with $n \geq 4$ edges has this property.
In Corollary~\ref{cor:area4} we conclude that 
our Pick-type theorem holds for all unimodular halfway-lattice polygons
with $n=4$ edges.


\begin{Remark}\labell{rk:opposite normals}
The result of Proposition~\ref{no opposite implies triangle} below
 -- that for every unimodular polygon with $n \geq 4$ edges,
its set of primitive inner normals contains two opposite vectors ---
appeared back in the 1970s in an algebraic-geometric context,
as a central step in the characterization 
of non-singular compact complex toric surfaces,
due to Miyake and Oda; see \cite[Page~40, Lemma~8.3]{oda}.
We include below a more leisurely proof
that we hope that some readers might find easier to follow
than the versions in the literature.
\end{Remark}

%

\begin{Notation} \labell{oC}
For any non-collinear vectors $u$ and $v$ in $\R^2$, we denote
$$ \bm{\mathring{C}}(u,v) 
   := \{ au+bv \ | \ a,b \in \R \text{ and } a,b > 0 \}\,.$$
See Figure~\ref{fig:oC}.
\end{Notation}

\begin{figure}[H]
\begin{tikzpicture}
\draw[thick,black,->] (0,0) -- (2,0) node[anchor=west] {$u$};
\draw[thick,black,->] (0,0) -- (0,2) node[anchor=south] {$v$};
\draw[,-,fill=gray!50,opacity=0.25,gray!50,opacity=0.25] 
(4,0) -- (0,0) -- (0,4) -- (2,3) -- (3,2) --cycle;
\draw[thick,black] (1.5,1.5) node[anchor=center]{$\oC(u,v)$};
\end{tikzpicture}
\caption{An open cone}
\label{fig:oC}
\end{figure}
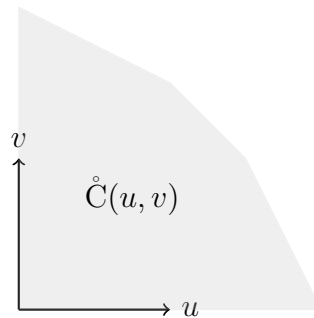

\begin{Lemma} \labell{new cones}
Let $v_1,v_2$ be vectors in $\R^2$ such that $\det(v_1,v_2) > 0$.  
For any $w,w',w'' \in \R^2$, 
\begin{enumerate}
\item \labell{newcones1}
If $w \in \oC(-v_1,-v_2)$, then $v_1 \in \oC(-v_2,-w)$.
\item \labell{newcones2}
If $w' \in \oC(v_2,-v_1)$, \ $w'' \in \oC(-v_2,v_1)$,
and $\det(w',w'')>0$, then $v_1,v_2 \in \oC(-w',-w'')$.
\end{enumerate}
See Figure \ref{fig:wwprime}.
\end{Lemma}

\begin{proof}[Sketch of proof]
For Part \eqref{newcones1},
write the assumption as a positive linear combination, and rearrange the terms
to obtain the conclusion.
For Part \eqref{newcones2}, 
write the assumptions as positive linear combinations,
and use elementary linear algebra to express $v_1,v_2$ in terms of $-w',-w''$.
\end{proof}


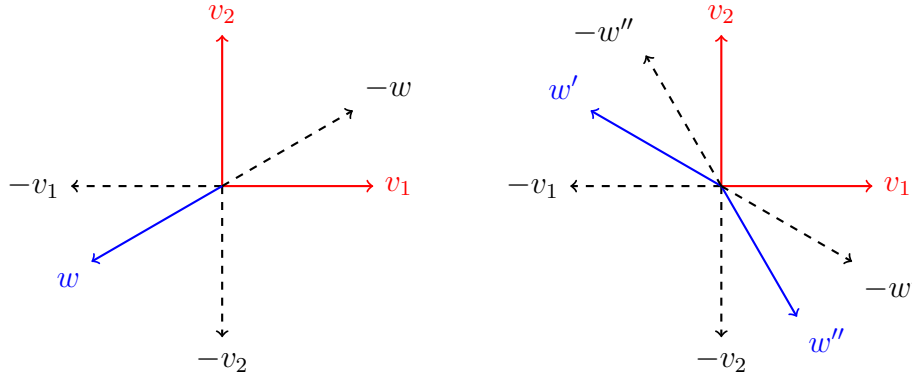
\begin{figure}[H]
\begin{tikzpicture}
\draw[thick,red,->] (0,0) -- (2,0) node[anchor=west] {$v_1$};
\draw[thick,dashed,->] (0,0) -- (-2,0) node[anchor=east] {$-v_1$};
\draw[thick,red,->] (0,0) -- (0,2) node[anchor=south] {$v_2$};
\draw[thick,dashed,->] (0,0) -- (0,-2) node[anchor=north] {$-v_2$};
\draw[thick,dashed,->] (0,0) -- (1.73,1) node[anchor=south west] {$-w$};
\draw[thick,blue,->] (0,0) -- (-1.73,-1) node[anchor=north east] {$w$};
\end{tikzpicture}
\qquad
\begin{tikzpicture}
\draw[thick,red,->] (0,0) -- (2,0) node[anchor=west] {$v_1$};
\draw[thick,dashed,->] (0,0) -- (-2,0) node[anchor=east] {$-v_1$};
\draw[thick,red,->] (0,0) -- (0,2) node[anchor=south] {$v_2$};
\draw[thick,dashed,->] (0,0) -- (0,-2) node[anchor=north] {$-v_2$};
\draw[thick,blue,->] (0,0) -- (-1.73,1) node[anchor=south east] {$w'$};
\draw[thick,blue,->] (0,0) -- (1,-1.73) node[anchor=north west] {$w''$};
\draw[thick,dashed,->] (0,0) -- (1.73,-1) node[anchor=north west] {$-w'$};
\draw[thick,dashed,->] (0,0) -- (-1,1.73) node[anchor=south east] {$-w''$};
\end{tikzpicture}
\caption{Configurations for Lemma~\ref{new cones}}
\label{fig:wwprime}
\end{figure}

\begin{Lemma}[\textbf{Unimodular linear pair\footnote{
A \emph{linear pair} is a pair of angle that have a common ray
and whose other rays are opposite.
See Figure~\ref{fig:between-linear_pair},
where the angles formed by $u,v$ and by $v,-u$ are also \emph{unimodular}
in the sense that the primitive generators to their edges
satisfy $\det(u,v) = \det(v,-u) = 1$.} lemma}]
\labell{lem:linear pair}
Let $\Sigma^1$ be the set of primitive inner normal vectors
to the edges of a unimodular polygon.
Let $u$, $v$ be lattice vectors such that ${\det(u,v)=1}$
and such that $v$ is not in $\Sigma^1$.
Then at least one of the sets 
$\Sigma^1 \cap \oC(u,v)$ and $\Sigma^1 \cap \oC(v,-u)$
is empty.  
\comment{\\ Consequence of Charlene's Lemma 3.22;
equivalent to Oda's \cite[p.~41, Sublemma~8.4]{oda}.} 
\end{Lemma}


\begin{proof}
Otherwise, let $w=au + bv$ with $a,b > 0$ be the last vector 
in $\Sigma^1 \cap \oC(u,v)$
in counterclockwise order,
and let $w'=cv + d(-u)$ with $c,d > 0$ be the first vector 
in $\Sigma^1 \cap \oC(v,-u)$
in counterclockwise order. See Figure \ref{fig:between-linear_pair}.
Because $v$ is not in $\Sigma^1$, 
the vectors $w$ and $w'$ are consecutive vectors in $\Sigma^1$
in counterclockwise order, so $\det(w,w') = 1$.
This contradicts the fact that 
$\det(w,w') = ac+bd$ with $a,b,c,d$ positive integers.

\begin{figure}[H]
\begin{tikzpicture}
\draw[thick,dashed,->] (0,0) -- (2,0) node[anchor=west] {$u$};
\draw[thick,dashed,->] (0,0) -- (0,2) node[anchor=south] {$v$};
\draw[thick,dashed,->] (0,0) -- (-2,0) node[anchor=east] {$-u$};
\draw[thick,blue,->] (0,0) -- (-1.2,1.7) node[anchor=south east] {$w'$};
\draw[thick,blue,->] (0,0) -- (1.0,1.8) node[anchor=south west] {$w$};
\draw (-1.4,0.8) node {\rotatebox{25}{$\iddots$}};
\draw (1.4,0.8) node {\rotatebox{-25}{$\ddots$}};
\end{tikzpicture}
\caption{Impossible configuration for Lemma~\ref{lem:linear pair}}
\label{fig:between-linear_pair}
\end{figure}
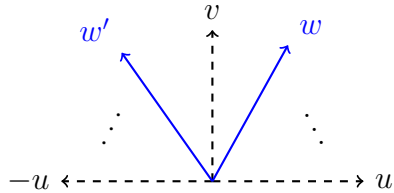
\end{proof}

We now recall the result from \cite[Page~40, Lemma~8.3]{oda},
and we give a self-contained proof (see Remark~\ref{rk:opposite normals}):

\begin{Proposition}
\labell{no opposite implies triangle}
Let $\Sigma^1$ be the set of primitive inner normal vectors
to the edges of a unimodular polygon.
Suppose that $\Sigma^1$ does not contain any two opposite vectors.
Then $\Sigma^1$ contains exactly three vectors.
\comment{Variation of Charlene's Proposition 3.23.}
\end{Proposition}

{
\begin{proof}
Enumerate the elements of $\Sigma^1$ in counterclockwise consecutive order
as $v_1,v_2,\ldots,v_n$,
and let $v_0 := v_1$ and $v_{n+1} := v_1$.
We proceed in three steps.

\smallskip\noindent\textbf{Step 1:}
\emph{For any two consecutive vectors $v_j,v_{j+1}$ in $\Sigma^1$,
either $\Sigma^1$ meets their opposite cone $\oC(-v_j,-v_{j+1})$
and does not meet their adjacent cones 
$\oC(v_{j+1},-v_j)$ nor $\oC(-v_{j+1},v_j)$,
or $\Sigma^1$ meets both of their adjacent cones
$\oC(v_{j+1},-v_j)$ and $\oC(-v_{j+1},v_j)$
and does not meet their opposite cone $\oC(-v_j,-v_{j+1})$.
}

\smallskip
Without loss of generality, we may focus on the consecutive vectors $v_1$
and $v_2$, and we may assume that $v_1=(1,0)$ and $v_2=(0,1)$.
Figure \ref{fig:dichotomy}
illustrates the two exclusive possibilities that we would like to obtain.

Let $v_3$ be the vector in $\Sigma^1$ that comes right after $v_2$,
and let $v_0$ be the vector in $\Sigma^1$ that comes right before $v_1$.
Because $\det(v_2,v_3)$ and $\det(v_0,v_1)$ are equal to $1$, hence positive,
$v_3$ lies to the left of the line through $\pm v_2$
and $v_0$ lies below the line through $\pm v_1$.
Because $\Sigma^1$ does not contain $-v_1$ nor $-v_2$,
we conclude that $\Sigma^1$ 
meets at least one of $\oC(v_2,-v_1)$ and $\oC(-v_1,-v_2)$
and at least one of $\oC(-v_1,-v_2)$ and $\oC(-v_2,v_1)$.

Finally, if $\Sigma^1$ meets $\oC(-v_1,-v_2)$,
then it cannot meet $\oC(v_2,-v_1)$ nor $\oC(-v_2,v_1)$,
by Lemma~\ref{lem:linear pair}
applied to $v_2$, $-v_1$, $-v_2$
and to $-v_1$, $-v_2$, $v_1$, respectively.  


\begin{figure}[H]
\begin{tikzpicture}
\draw[thick,red,->] (0,0) -- (2,0) node[anchor=west] {$v_{1}$};
\draw[thick,dashed,->] (0,0) -- (-2,0) node[anchor=east] {$-v_{1}$};
\draw[thick,red,->] (0,0) -- (0,2) node[anchor=south] {$v_{2}$};
\draw[thick,dashed,->] (0,0) -- (0,-2) node[anchor=north] {$-v_{2}$};
\draw[thick,blue,->] (0,0) -- (-1.73,1) node[anchor=south east] {};
\draw[thick,blue,->] (0,0) -- (1,-1.73) node[anchor=north west] {};
\end{tikzpicture}
\qquad
\begin{tikzpicture}
\draw[thick,red,->] (0,0) -- (2,0) node[anchor=west] {$v_{1}$};
\draw[thick,dashed,->] (0,0) -- (-2,0) node[anchor=east] {$-v_{1}$};
\draw[thick,red,->] (0,0) -- (0,2) node[anchor=south] {$v_{2}$};
\draw[thick,dashed,->] (0,0) -- (0,-2) node[anchor=north] {$-v_{2}$};
\draw[thick,blue,->] (0,0) -- (-1.73,-1) node[anchor=south east] {};
\end{tikzpicture}
\caption{Dichotomy for Step 1} 
\label{fig:dichotomy}
\end{figure}
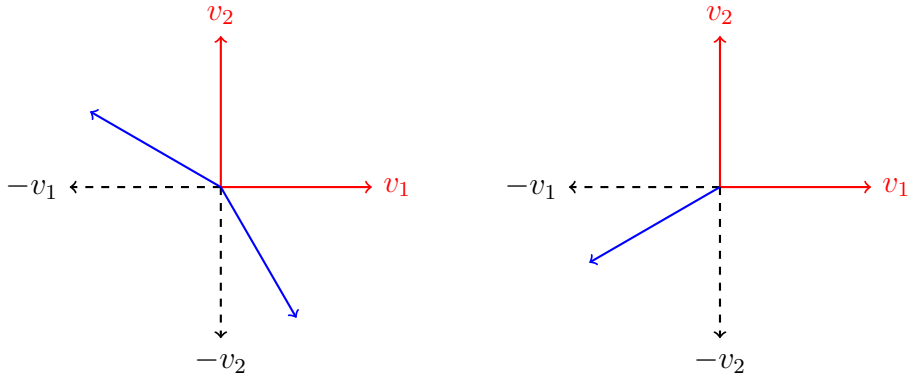

\smallskip\noindent\textbf{Step 2:}
\emph{There exist two consecutive vectors $v_j,v_{j+1}$ in $\Sigma^1$
such that $\Sigma^1$ meets their opposite cone $\oC(-v_j,-v_{j+1})$.}


\smallskip
If $\Sigma^1$ meets $\oC(-v_1,-v_{2})$, then the consecutive vectors 
$v_1$, $v_2$ 
are as required.
Otherwise, let $v_j$ be the last vector in $\Sigma^1$ before $-v_1$,
and let $v_{j+1}$ be the first vector in $\Sigma^1$ after $-v_2$;
see Figure~\ref{fig:oppositeVectors1}.
%
By Step 1,
$v_j \in \oC(v_2,-v_1)$ and $v_{j+1} \in \oC(-v_2,v_1)$.
By Lemma~\ref{new cones}\eqref{newcones2}, 
$v_1,v_2 \in \oC(-v_j,-v_{j+1})$.
So the consecutive vectors $v_j$ and $v_{j+1}$ 
are as required.

\begin{figure}[H]
\begin{tikzpicture}
\draw[thick,red,->] (0,0) -- (2,0) node[anchor=west] {$v_{1}$};
\draw[thick,dashed,->] (0,0) -- (-2,0) node[anchor=east] {$-v_{1}$};
\draw[thick,red,->] (0,0) -- (0,2) node[anchor=south] {$v_{2}$};
\draw[thick,dashed,->] (0,0) -- (0,-2) node[anchor=north] {$-v_{2}$};
\draw[thick,blue,->] (0,0) -- (-1.73,1) node[anchor=south east] {$v_{j}$};
\draw[thick,dashed,->] (0,0) -- (1.73,-1) node[anchor=north west] {$-v_{j}$};
\draw[thick,blue,->] (0,0) -- (1,-1.73) node[anchor=north west] {$v_{j+1}$};
\draw[thick,dashed,->] (0,0) -- (-1,1.73) node[anchor=south east] {$-v_{j+1}$};
\end{tikzpicture}
\caption{Configuration for Step 2} \label{fig:oppositeVectors1}
\end{figure}
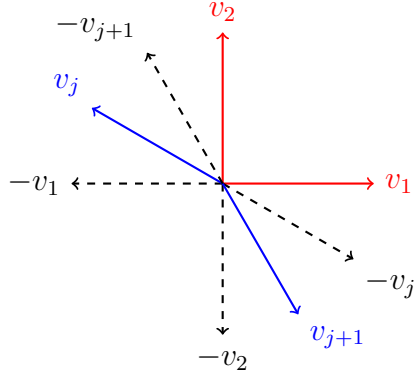

\smallskip\noindent\textbf{Step 3:}
\emph{$\Sigma^1$ contains exactly three vectors.}


\smallskip
By Step 2,
without loss of generality we may assume 
that $\Sigma^1$ meets $\oC(-v_{1},-v_{2})$;
see Figure~\ref{fig:oppositeVectors9}.
By Step 1,
it follows that $\Sigma^1$ 
does not meet $\oC(v_2,-v_1)$ nor $\oC(-v_2,v_1)$.

Let $v_3$ be the next vector in $\Sigma^1$ after $v_2$.
Then $v_3 \in \oC(-v_1,-v_2)$.
By Lemma \ref{new cones}\eqref{newcones1}, 
this implies that $v_1 \in \oC(-v_2,-v_3)$.

By 
Lemma~\ref{lem:linear pair}
for $v_3,-v_2,-v_3$,
and because $\Sigma^1$ meets $\oC(-v_2,-v_3)$ (in the vector $v_1$),
we conclude that $\Sigma^1$ does not meet $\oC(v_3,-v_2)$.
But $\Sigma^1$ also does not contain $-v_2$ nor meets $\oC(-v_2,v_1)$.
So the next vector in $\Sigma^1$ after $v_3$ is $v_1$,
and so $\Sigma^1 = \{v_1,v_2,v_3\}$.

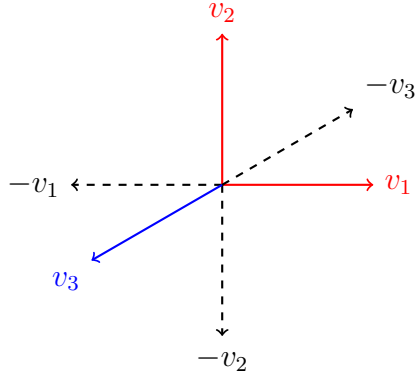
\begin{figure}[H]
\begin{tikzpicture}
\draw[thick,red,->] (0,0) -- (2,0) node[anchor=west] {$v_{1}$};
\draw[thick,dashed,->] (0,0) -- (-2,0) node[anchor=east] {$-v_{1}$};
\draw[thick,red,->] (0,0) -- (0,2) node[anchor=south] {$v_{2}$};
\draw[thick,dashed,->] (0,0) -- (0,-2) node[anchor=north] {$-v_{2}$};
\draw[thick,blue,->] (0,0) -- (-1.73,-1) node[anchor=north east] {$v_{3}$};
\draw[thick,dashed,->] (0,0) -- (1.73,1) node[anchor=south west] {$-v_{3}$};
\end{tikzpicture}
\caption{Configuration for Step 3} \label{fig:oppositeVectors9}
\end{figure}
\end{proof}
}

\begin{Corollary} \labell{prop:hirzebruchTrapezoidCongruence}
Every unimodular quadrilateral is lattice congruent to a standard trapezoid.
\comment{Charlene's Proposition 3.18.}
\end{Corollary}

\begin{proof}
Let $P$ be a unimodular quadrilateral.
Let $\Sigma^1$ be the set of primitive inner normal vectors to its edges.
By Proposition~\ref{no opposite implies triangle},
$\Sigma^1$ contains two opposite vectors.
By Proposition~\ref{prop:parallel-then-standard},
$P$ is lattice congruent to a standard trapezoid.
\end{proof}

\begin{Corollary} \labell{cor:area4}
Let $P$ be a unimodular halfway-lattice quadrilateral.
Let $A$ be its area, and let $I$ be the number of lattice points
in its interior.  Then $$A = I.$$
\end{Corollary}

\begin{proof}
By Proposition~\ref{no opposite implies triangle},
the set $\Sigma^1$ of primitive normal vectors
to the edges of $P$ contains two opposite vectors.
So $P$ is a trapezoid, and the result follows
from Proposition~\ref{prop:area4-parallel}.
\end{proof}

\section{Unimodular halfway-lattice polygons with five or more edges}
\labell{sec:five}

To prove our Pick-type formula in the general case, 
we will use an inductive argument
that relies on the well-known fact that, up to lattice congruence, 
every unimodular polygon can be obtained from a standard triangle
or a standard trapezoid by a sequence of corner-choppings.
See Definition~\ref{def:chopping} and Lemma~\ref{lem:fillingChopping} below.

\begin{Remark}\labell{rk:blowups}
This well-known fact
follows from a property of the primitive normals to the edges
 --- Proposition \ref{prop:oppositeVectors} --
that appeared in Oda's \cite[Page~42]{oda},
in the same algebraic-geometric context
that we mentioned in Remark~\ref{rk:opposite normals}.
We provide
a self-contained proof of this property that is similar to the argument in 
\cite[Proof of Theorem~6.1]{karshon:periodic}.
\end{Remark}


We will start with a simple lemma: 

\begin{Lemma} \labell{lem:MicheleAudin}
Let $v_1$, $v_2$, and $v_3$ be vectors in $\Z^2$.
Suppose that $\det (v_1, v_2) = 1$.  
Then $\det ( v_2, v_3 ) = 1$ 
if and only if there exists an integer $m$ such that $v_1 + v_3 = m v_2$.
\comment{ \\ Charlene's Lemma 3.6.}
\end{Lemma}

\begin{proof}
Because $\det ( v_1, v_2 ) = 1$, 
up to lattice congruence, we may assume
that $\ds v_1 = \binom{1}{0}$ and $\ds v_2 = \binom{0}{1}$.
Write $\ds v_3 = \binom{x_3}{y_3}$.
Then $x_3 = -\det ( v_2, v_3 ) = -1$,
and so $\ds v_1 + v_3 = \binom{0}{y_3} = mv_2$ with $m=y_3$.
\end{proof}

Next, we recall the result from \cite[Page~42]{oda},
and we give a self-contained proof (see Remark~\ref{rk:blowups}):

\begin{Proposition} \labell{prop:oppositeVectors}
Let $P$ be a unimodular polygon with $n \geq 5 $ edges. 
Let $v_{1},\dots,v_{n}$ be the primitive inner normal vectors 
to the edges of $P$, in counterclockwise consecutive order, 
and let $v_0:=v_n$ and $v_{n+1} := v_1$.
Then there exists $1 \leq j \leq n$ such that $v_j = v_{j-1} + v_{j+1}$.
\comment{\\ Charlene's Proposition 3.23 Part 4(d)}
\end{Proposition}

\begin{proof}
By Proposition~\ref{no opposite implies triangle},
there are two opposite vectors $v_{j}$ and $v_{k}$ in $\Sigma^{1}$. 
Because $n \geq 5$, we can re-index the vectors 
such that the opposite vectors are $v_1$ and $v_i$ 
and such that at least two of the vectors lie after $v_1$ and before $v_i$.
Up to lattice congruence, we may further assume that $\ds v_1 = \binom{1}{0}$
and $\ds v_i = \binom{-1}{0}$.
Moreover, $y_2 = \det(v_1,v_2) = 1$ and $y_{i-1} = \det(v_{i-1},v_{i}) = 1$.
See Figure~\ref{fig:oppositeVectors_part4}.

\begin{figure}[H]
\begin{tikzpicture}
\draw[thick,->] (0,0) -- (2,0) node[anchor=west] 
{$y_1=0$};
\draw[thick,->] (0,0) -- (0,2) node[anchor=south] {$(x_j,y_j)$};
\draw[thick,->] (0,0) -- (-2,0) node[anchor=east] 
{$y_i=0$};
\draw[thick,->] (0,0) -- (1.73,1) node[anchor=south west] 
{$y_{2}=1$};
\draw[thick,->] (0,0) -- (-1.73,1) node[anchor=south east] 
{$y_{i-1}=1$};
\draw (-0.5,1) node {$\dots$};
\draw (0.5,1) node {$\dots$};
\end{tikzpicture}
\caption{} \label{fig:oppositeVectors_part4}
\end{figure}

Write $\ds v_j = \binom{x_j}{y_j}$.
Then $y_1=y_i=0$, and $y_2,\ldots,y_{i-1}$ are positive integers. 
By this and Lemma~\ref{lem:MicheleAudin},
for each $1 < j < i$ there exists a positive integer $m_j$ such that 
\begin{equation} \labell{mj}
v_{j-1} + v_{j+1} = m_j v_j
\quad \text{ and } \quad
m_j = \frac{y_{j-1}+y_{j+1}}{y_j}.
\end{equation}

Because $ x_{j-1}y_{j} - x_{j}y_{j-1} = \det(v_{j-1},v_{j}) = 1$,
any two consecutive integers $y_{j-1}$ and $y_j$ are relatively prime.

If $y_2 = \ldots = y_{i-1} = 1$, then $y_1=0$ and $y_2=y_3=1$.
So in \eqref{mj} we obtain $m_2=1$, and we are done.
%
Otherwise, the maximal $y_j$ among $y_2,\ldots,y_{i-1}$ satisfies $y_j \geq 2$.
Since $y_j$ is maximal,
\begin{equation} \labell{ineq}
y_j \geq y_{j-1} \quad \text{ and } \quad y_j \geq y_{j+1}.
\end{equation}
But $y_j$ is relatively prime to $y_{j-1}$ and to $y_{j+1}$
and is $\geq 2$, so the inequalities in \eqref{ineq} must be strict;
adding them up and dividing by $y_j$, we obtain
$$ \frac{y_{j-1}+y_{j+1}}{y_j} < 2\,. $$
The left term, being a positive integer --- namely, $m_j$ ---
and less than $2$, must be equal to~$1$.
So $m_j=1$, and $v_{j-1} + v_{j+1} = v_j$.
\end{proof}
}

\begin{Definition}
\labell{def:chopping}
Let $Q$ be a unimodular polygon with $n-1$ edges. 
Let $H$ be a closed half-plane.
If $P=Q\cap H$ is a unimodular polygon with $n$ edges,
we say that $P$ is obtained by {\bf corner-chopping} $Q$.
See Figure~\ref{fig:chopping}.
\comment{Charlene's Definition 3.25}
\end{Definition}

{
\begin{figure}[H]
\centering
  \begin{tikzpicture}
 
    \coordinate (A) at (6,0);
    \coordinate (B) at (0,0);
    \coordinate (C) at (0,6);
    \coordinate (D) at (3,6);
 
    \coordinate (B1) at (2,0);
    \coordinate (B2) at (0,2);
 
    \draw [,-,fill=gray!50, opacity=0.25,gray!50, opacity=0.25] (A) -- (B) -- (C) -- (D) -- (A)--cycle;
 
    \draw [,-,fill=gray!50, opacity=0.25,gray!50, opacity=0.25] (A) -- (B1) -- (B2) -- (C) -- (D) -- (A)--cycle;
    \draw [thick,-,black] (A) -- (B) -- (C) -- (D) -- (A) node [above left] {};
 
    \coordinate (Q) at (3.5,5);
    \draw [thick,-,black] (Q) node [above right] {$Q$};
 
    \coordinate (P) at (2.5,3);
    \draw [thick,-,black] (P) node [above right] {$P$};

    \coordinate (R) at (.25,.25);
    \draw [thick,-,black] (R) node [above right] {$R$};
 
    \coordinate (vline1a) at (2,0);
    \coordinate (vline1b) at (0,2);
    \draw [thick,dashed,black] (vline1a) -- (vline1b) node [] {};
 
    \coordinate (vline2a) at (0,2);
    \coordinate (vline2b) at (-1,3);
    \draw [thick,-,black] (vline2a) -- (vline2b) node [below left] {$\partial H_{i}$};
 
    \coordinate (hline1a) at (3,-1);
    \coordinate (hline1b) at (2,0);
    \draw [thick,-,black] (hline1a) -- (hline1b) node [] {};
 
    \coordinate (v1a) at (4,0);
    \coordinate (v1b) at (4,1);
    \coordinate (v2a) at (1,1);
    \coordinate (v2b) at (2,2);
    \coordinate (v3a) at (0,4);
    \coordinate (v3b) at (1,4);
    \draw [ultra thick,-latex,blue] (v1a) -- (v1b) node [right] {$v_{i+1}$};
    \draw [ultra thick,-latex,blue] (v2a) -- (v2b) node [above left] {$v_{i}$};
    \draw [ultra thick,-latex,blue] (v3a) -- (v3b) node [above] {$v_{i-1}$};
 
  \end{tikzpicture}
  \caption{\comment{Charlene's Figure 14 -- modified}} \label{fig:chopping}
\end{figure}
}

\begin{Lemma}\labell{lem:fillingChopping}
Let $P$ be a unimodular halfway-lattice polygon with $n\geq 5$ edges. 
Then there exists a unimodular halfway-lattice polygon $Q$ with $n-1$ edges 
such that $P$ is obtained by corner-chopping~$Q$.
Moreover, the polygon $Q$ decomposes as a union $Q = P \cup R$ 
where $R$ is a halfway-lattice triangle 
and $P$ and $R$ meet along a common edge.
See Figure~\ref{fig:chopping}.
\comment{\\ Charlene's Lemmas 3.26 and 3.27.}
\end{Lemma}

\begin{proof}
Let $H_{1},\dots,H_{n}$ be the smallest collection of closed half-planes 
such that $P = H_{1}\cap \dots \cap H_{n}$. 
Since $P$ is a halfway-lattice polygon, 
$\del H_1 , \ldots, \del H_n$ are halfway-lattice lines. 
Let $v_{1},\dots,v_{n}$ be the primitive inner normal vectors 
to $\del H_1, \ldots, \del H_n$; without loss of generality,
we may assume that they are enumerated in counterclockwise
consecutive order.  Let $v_0 := v_n$ and $v_{n+1} := v_1$.

By Proposition~\ref{prop:oppositeVectors}, $v_{i-1} + v_{i+1} = v_{i}$ 
for some $1\leq i \leq n$. 
Because $\det(v_{i-1},v_i) = 1$ and $v_i = v_{i-1}+v_{i+1}$, we have
$$ \det (v_{i-1},v_{i+1}) = \det (v_{i-1},v_i-v_{i-1}) = 1.$$
So up to lattice congruence we may assume without loss of generality
that $\ds v_{i-1} = \binom{1}{0}$ and $\ds v_{i+1} = \binom{0}{1}$, 
and that $\ds v_i = \binom{1}{1}$.  See Figure~\ref{fig:chopping}.
Moreover,
up to translation, we may furthermore assume that $H_{i-1}$
and $H_{i+1}$ are, respectively, the half-planes given by the equations
$\{ x \geq \frac{1}{2} \}$ and $\{ y \geq \frac{1}{2} \}$ in the $xy$-plane.
The half-plane $H_i$ is then given by $\{x+y \geq m+\frac{1}{2}\}$ 
for some $m \in \N$.

The intersection $Q := H_{1}\cap\dots H_{i-1}\cap H_{i+1}\cap\dots\cap H_{n}$
is then a convex polygon with $n-1$ edges,
whose edges lie on the halfway-lattice lines
$\del H_1,\ldots, \del H_{i-1}, \del H_{i+1}, \ldots, \del H_n$.
Thus, $Q$ is a halfway-lattice polygon.
The vectors $v_1,\ldots,v_{i-1},v_{i+1},\ldots,v_n$ 
are the primitive inner normal vectors to its edges,
in counterclockwise consecutive order.
We have shown that $\det (v_{i-1},v_{i+1}) = 1$.
Because $P$ is unimodular,
$\det (v_{j-1},v_j) = 1$ for all $1 \leq j \leq i-1$
and for all $i+2 \leq j \leq n$.
Thus, $Q$ is unimodular.
\end{proof}


\begin{Theorem}[The main theorem]\labell{thm:mainTheorem}
Let $P$ be a unimodular halfway-lattice polygon.
Let $A$ be its area, let $I$ be the number of lattice points 
in its interior, and let $V$ be its number of vertices.  Then 
\comment{Charlene's Theorem 3.29}
\begin{equation} \labell{Chu-formula}
A = I - \dfrac{V}{8} + \dfrac{1}{2}.
\end{equation}
\end{Theorem}

\begin{proof}
We argue by induction on the number $n$ of edges of $P$
(which is equal to the number $V$ of vertices).

\smallskip
\emph{Base Case:} 
If $n=3$, then the right hand side of \eqref{Chu-formula}
equals $\ds I + \frac{1}{8}$, 
and the result of the theorem follows from Proposition~\ref{prop:area3}.
If $n=4$, then the right hand side of \eqref{Chu-formula} equals $I$, 
and the result of the theorem follows from Corollary~\ref{cor:area4}.

\smallskip
\emph{Induction Step:} 
Suppose that $n \geq 5$ and that the result of the theorem 
holds for unimodular halfway-lattice polygons with fewer than $n$ edges. 
Let $P$ be a unimodular halfway-lattice polygon with $n$ edges.
By Lemma~\ref{lem:fillingChopping}, there exists a unimodular halfway-lattice
polygon $Q$ with $n-1$ edges such that $P$ is obtained by corner-chopping $Q$.
Moreover, the polygon $Q$ decomposes as a union $Q = P \cup R$ 
where $R$ is a halfway-lattice triangle 
and $P$ and $R$ meet along a common edge, as in Figure~\ref{fig:chopping}.
Denote by $A_Q$ and $A_R$ the areas of $Q$ and $R$, respectively,
and denote by $I_Q$ and $I_R$ the number of lattice points in their interiors.
Then 
$$ A_Q = A + A_R \quad \text{ and } \quad I_Q = I + I_R.$$
By the induction hypothesis, 
$$A_Q = I_Q - \frac{n-1}{8} + \frac{1}{2}.$$
By Proposition~\ref{prop:area3}, 
$$A_R = I_R + \frac{1}{8}.$$ 
So 
$$ A \, = \,  A_Q - A_R 
     \, = \, \left(I_Q - \frac{n-1}{8} + \frac{1}{2} \right) 
           - \left(I_R + \frac{1}{8} \right) 
     \, = \, \big(I_Q - I_R\big) - \frac{n}{8} + \frac{1}{2}
     \, = \, I - \frac{V}{8} + \frac{1}{2}\,,$$
as required.
\end{proof}

{\ifworkmode {\newpage} \else {\end{document}} \fi}

\begin{itemize}
\item
Alexander Caviedes-Castro and his student write ``halfway polytopes''.
\item
Should we define ``fan'' and ``unimodular fan''? No.
\end{itemize}

\bigskip

Definition: distinct primitive vectors $v_1,v_2, v_3$ 
are ordered \textbf{counterclockwise}
if at least two of $\det(v_1, v_2)$, $\det(v_2, v_3)$, $\det(v_3, v_1)$
are strictly positive and \textbf{clockwise}
if at least two of $\det(v_1, v_2)$, $\det(v_2, v_3)$, $\det(v_3, v_1)$
are strictly negative. 
\\ Lemma: any three distinct primitive vectors $v_1,v_2,v_3$
are ordered either counterclockwise or clockwise, and not both.
\\ (Definition: for $k \geq 4$, primitive vectors $v_1,v_2,\ldots,v_k$ 
are ordered \textbf{counterclockwise}
if, setting $v_0=v_k$ and $v_{k+1}=v_1$,
for any $1 \leq i \leq k$ the vectors $v_{i-1},v_i,v_{i+1}$
are distinct and ordered counterclockwise.)
\\ (Lemma: something about $v_1,v_2,v_3,v_4$.)

\bigskip

Our result should be a special case of more general results of Beck-Robins.

\bigskip

In connection with Proposition~\ref{no opposite implies triangle}
(existence of opposite vectors):
 \\ Miyake-Oda \cite[Theorem~11]{miyake-oda}.
 \\ Oda's later book: \cite[Theorem 1.28]{oda-book}.
 \\ Fulton \cite[Section~2.5 and Notes to Chapter 2]{fulton}.
 \\ check Barvinok's book and/or Beck+Robins?

The term ``linear pair'' is NOT commonly used in K--12 geometry.


\begin{thebibliography}{99}



\bibitem{audin}
Mich\`ele Audin, ``Torus actions on symplectic manifolds'',
Second revised edition.
Progress in Mathematics, \textbf{93}, Birkh\"auser Verlag, Basel, 2004.

\bibitem{beck-robins}
Matthias Beck and Sinai Robins,
``Computing the continuous discretely. Integer-point enumeration
in polyhedra.''
Second edition.  Undergraduate Texts in Mathematics.
Springer, New york, 2015.

\bibitem{brion-vergne}
Michel Brion and Mich\`ele Vergne,
\emph{Lattice points in simple polytopes}, 
J.\ Amer.\ Math.\ Soc.\ \textbf{10} (1997), no.~2, 371--392.

\bibitem{camarneiro-CdS-thesis}
Jo\~ao Camarneiro, \emph{Toric real loci}. 
M.Sc.\ thesis, supervised by Ana Cannas da Silva, E.T.H., 2024


\bibitem{duarte-ACC}
Natalia Juliana Lozano Duarte, \emph{Khovanskii-Pukhlikov formula
for half-way polytopes}.  M.Sc.\ thesis, supervised by 
Alexander Caviedes Castro, Pontificia Universidad Javeniana, 2025.


\bibitem{guillemin}
Victor Guillemin, \emph{Riemann-Roch for toric orbifolds},
J.\ Diff.\ Geom.\ \textbf{45} (1997), 53--73.

\bibitem{grossberg-karshon}
Michael Grossberg and Yael Karshon,
\emph{Equivariant index and the moment map for completely integrable
torus actions}, Adv.\ Math.\ \textbf{133} (1998), 185--223.

\bibitem{karshon:periodic}
Yael Karshon, \emph{Periodic Hamiltonian flows on four dimensional manifolds}.
Mem.\ Amer.\ Math.\ Soc.\ \textbf{672}, 1999.

\bibitem{karshon-sternberg-weitsman-PNAS}
Yael Karshon, Shlomo Sternberg, and Jonathan Weitsman,
\emph{The Euler-Maclaurin formula for simple integral polytopes}, 
Proc.\ Natl.\ Acad.\ Sci.\ USA \textbf{100}(2003), no.~2, 426--433.

\bibitem{karshon-sternberg-weitsman-remainder}
Yael Karshon, Shlomo Sternberg, and Jonathan Weitsman,
\emph{Euler-Maclaurin with remainder for a simple integral polytope}, 
Duke Math.\ J.\ \textbf{130} (2005), no.~3, 401--434.

\bibitem{karshon-sternberg-weitsman-exact}
Yael Karshon, Shlomo Sternberg, and Jonathan Weitsman,
\emph{Exact Euler-Maclaurin formulas for simple lattice polytopes}, 
Adv.\ Appl.\ Math.\ \textbf{39} (2007), no.~1, 1--50.

\bibitem{khovanskii-pukhlikov}
Askold G.\ Khovanskii and Aleksandr V.\ Pukhlikov,
\emph{A Riemann-Roch theorem for integrals and sums of quasipolynomials
over virtual polytopes}, Algebra and Analysis \textbf{4} (1992), 188--216.
Translation: St.\ Petersburg Math.\ J.\ \textbf{4} (1993), no.~4, 789--812.



\bibitem{oda}
Tadao Oda, \emph{Torus embeddings and applications
(based on joint work with Katsuya Miyake)}.
Tata Institute of Fundamental Research Lectures on Mathematics and Physics
\textbf{57}, Springer-Verlag, New-York, 1978.


\bibitem{pick}
Georg Pick, \emph{Geometrisches zur Zahlenlehre}.
{Sitzungsberichte des deutschen naturwissenschaftlich--medicinischen 
Vereines f\"{u}r B\"{o}hmen ``Lotos'' in Prag.}
(Neue Folge). \textbf{19} (1899), pp.~311--319.

\bibitem{vaughan}
Jennifer Vaughan, 
\emph{Equivariant metaplectic-c prequantization of symplectic manifolds
with Hamiltonian torus actions}, arXiv:1702.01476 [math.SG].

\end{thebibliography}
\end{document}